\documentclass[final,nomarks]{dmtcs-episciences}
\usepackage[utf8]{inputenc}
\usepackage[T1]{fontenc}
\usepackage{graphicx,amsmath,amsthm,amsfonts,amssymb,microtype,thmtools,underscore,mathtools,MnSymbol,wrapfig}
\usepackage[shortlabels]{enumitem}
\setlist[itemize]{topsep=0.5ex,itemsep=0ex,parsep=0.4ex}
\setlist[enumerate]{topsep=0.5ex,itemsep=0ex,parsep=0.4ex}
\usepackage[dvipsnames,svgnames,table]{xcolor}
\hypersetup{
unicode=true,
colorlinks=true,
breaklinks=true,
linkcolor={blue},
citecolor={black},
urlcolor={black},
pdftitle={Treewidth 2 in the Planar Graph Product Structure Theorem},
pdfauthor={Marc Distel, Kevin Hendrey, Nikolai Karol, David~R.~Wood, Jung Hon Yip}
} 
\usepackage[noabbrev,capitalise]{cleveref}
\crefname{lem}{Lemma}{Lemmas}
\crefname{thm}{Theorem}{Theorems}
\crefname{cor}{Corollary}{Corollaries}
\crefname{prop}{Proposition}{Propositions}
\crefname{conj}{Conjecture}{Conjectures}
\crefname{open}{Open Problem}{Open Problems}
\crefname{claim}{Claim}{Claims}
\crefformat{equation}{(#2#1#3)}
\Crefformat{equation}{Equation #2(#1)#3}

\newcommand{\defn}[1]{\textcolor{Maroon}{\emph{#1}}}

\usepackage[longnamesfirst,numbers,sort&compress]{natbib}
\makeatletter
\def\NAT@spacechar{~}
\makeatother
\renewcommand{\geq}{\geqslant}
\renewcommand{\leq}{\leqslant}

\newcommand{\StrongProd}{\mathbin{\boxtimes}}

\DeclareMathOperator{\tw}{tw}
\DeclareMathOperator{\pw}{pw}
\newcommand{\subsetsim}{\mathrel{\substack{\textstyle\subset\\[-0.6ex]\textstyle\sim\\[-0.4ex]}}}
\renewcommand{\thefootnote}{\fnsymbol{footnote}}
\theoremstyle{plain}
\newtheorem{thm}{Theorem}
\newtheorem{lem}[thm]{Lemma}
\newtheorem{cor}{Corollary}
\newtheorem{obs}{Observation}
\newcommand{\PP}{\mathcal{P}}
\newcommand{\QQ}{\mathcal{Q}}
\newcommand{\NN}{\mathbb{N}}

\author[Marc Distel, Kevin Hendrey, Nikolai Karol, David~R.~Wood, Jung Hon Yip]{Marc Distel  \quad Kevin Hendrey  \quad Nikolai Karol  \\ \centering David~R.~Wood \quad Jung Hon Yip}
\title[Treewidth 2 in the Planar Graph Product Structure Theorem]{Treewidth 2 in the Planar Graph Product Structure Theorem}
\affiliation{School of Mathematics, Monash   University, Melbourne, Australia}
\begin{document}
\publicationdata{vol. 27:2}{2025}{8}{10.46298/dmtcs.14785}{2024-11-18; 2024-11-18; 2025-03-10}{2025-03-11}
\maketitle

\begin{abstract}
\medskip We prove that every planar graph is contained in $H_1\boxtimes H_2\boxtimes K_2$ for some graphs $H_1$ and $H_2$ both with treewidth 2. This resolves a question of Liu, Norin and Wood [arXiv:2410.20333]. We also show this result is best possible in the following sense: for any $c \in \NN$, there is a planar graph $G$ such that for any tree $T$ and graph $H$ with $\tw(H) \leq 2$, $G$ is not contained in $H \boxtimes T \boxtimes K_c$.
\end{abstract}

\renewcommand{\thefootnote}
{\arabic{footnote}}

\section{\boldmath Introduction}
\label{Introduction}

Graph product structure theory describes graphs in complicated graph classes as subgraphs of products of graphs in simpler graph classes, typically with bounded treewidth or bounded pathwidth. As defined in \cref{Treewidth}, the treewidth of a graph $G$, denoted by \defn{$\tw(G)$}, is the standard measure of how similar $G$ is to a tree. As illustrated in \cref{ProductExample}, the \defn{strong product} $A \StrongProd B$ of
graphs $A$ and $B$ has vertex-set 
$V(A) \times V(B)$,
\begin{wrapfigure}{r}{60mm}
\centering
\vspace*{-1ex}
\includegraphics{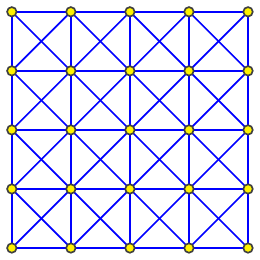}
\vspace*{-1ex}
\caption{\label{ProductExample} Strong product of paths.}
\end{wrapfigure}
where distinct vertices $(v,x), (w,y)$ are adjacent if:
\begin{itemize}
\item $v = w$ and $xy \in E(B)$, or 
\item $x = y$ and $vw \in E(A)$, or
\item $vw\in E(A)$ and $xy \in E(B)$.
\end{itemize}

\smallskip
The following Planar Graph Product Structure Theorem is the classical example of a graph product structure theorem. Here, a graph $H$ is \defn{contained} in a graph $G$ if $H$ is isomorphic to a subgraph of $G$, written \defn{$H \subsetsim G$}. 

\clearpage
\begin{thm}
\label{PGPST} 
For every planar graph $G$:
\begin{enumerate}[(a)]
\item $G\subsetsim H \StrongProd P$ for some graph $H$ with $\tw(H)\leq 6$ and path $P$ \textup{\citep{UWY22}},
\item $G\subsetsim H \StrongProd P \StrongProd K_2$ for some graph $H$ with $\tw(H)\leq 4$ and path $P$ \textup{\citep{UWY22}},
\item $G\subsetsim H \StrongProd P \StrongProd K_3$ for some graph $H$ with $\tw(H)\leq 3$ and path $P$ \textup{\citep{DJMMUW20}}.
\end{enumerate}
\end{thm}

\citet{DJMMUW20} first proved \cref{PGPST}(a) with $\tw(H)\leq 8$. In follow-up work, \citet{UWY22} improved the bound to $\tw(H)\leq 6$. Part (b) is due to Dujmovi\'c, and is presented in \citep{UWY22}. Part (c) is in the original paper of  \citet{DJMMUW20}. \citet{ISW24} gave a new proof of part (c).

\cref{PGPST} provides a powerful tool for studying questions about planar graphs, by reducing to graphs of bounded treewidth. Indeed, this result has been the key for resolving several open problems regarding queue layouts~\cite{DJMMUW20}, nonrepetitive colourings~\cite{DEJWW20}, centred colourings~\cite{DFMS21}, adjacency labelling schemes~\cite{GJ22,BGP22,EJM23,DEGJMM21}, twin-width~\cite{BKW,KPS24,JP22}, infinite graphs~\cite{HMSTW}, and comparable box dimension~\cite{DGLTU22}. In several of these applications, because the dependence on $\tw(H)$ is often exponential, the best bounds are obtained by applying the 3-term product in \cref{PGPST}(c). 

The $\tw(H)\leq 3$ bound in \cref{PGPST}(c) is best possible in any result saying that every planar graph is contained $H \StrongProd P \StrongProd K_c$ where $P$ is a path (see \citep{DJMMUW20}). \citet{LNW} relaxed the assumption that $P$ is a path, and studied products of two graphs of bounded treewidth. They asked whether  every planar graph is contained in $H_1 \StrongProd H_2 \StrongProd K_c$ for some graphs $H_1$ and $H_2$ with $\tw(H_1)\leq 2$ and $\tw(H_2)\leq 2$. We answer this question in the affirmative. 

\begin{thm}
\label{Planar222} 
Every planar graph $G$ is contained in $H_1 \StrongProd H_2 \StrongProd K_2$ for some graphs $H_1$ and $H_2$ with $\tw(H_1)\leq 2$ and $\tw(H_2)\leq 2$.
\end{thm}

We actually prove a strengthening of \cref{Planar222} that holds for a more general class of graphs $G$, and with a more precise statement about the structure of $H_1$ and $H_2$; see \cref{Apex} below. We also show that \cref{Planar222} is best possible in the following sense. 

\begin{thm}
\label{PlanarLowerBound}
For any integer $c\geq 1$ there is a planar graph $G$ such that for any tree $T$ and graph $H$ with $\tw(H)\leq 2$, $G$ is not contained in $H \StrongProd T \StrongProd K_c$. 
\end{thm}

We conclude this introduction by mentioning an open problem: Does \cref{Planar222} hold with
$H_1\boxtimes H_2\boxtimes K_2$ replaced by $H_1\boxtimes H_2$? 

\section{Treewidth}
\label{Treewidth}

We consider finite simple undirected graphs $G$ with vertex-set $V(G)$ and edge-set $E(G)$. For a tree $T$ with $V(T)\neq\emptyset$, a \defn{$T$-decomposition} of a graph $G$ is a collection $(B_x:x \in V(T))$ such that:
\begin{itemize}
    \item $B_x\subseteq V(G)$ for each $x\in V(T)$, 
    \item for every edge ${vw \in E(G)}$, there exists a node ${x \in V(T)}$ with ${v,w \in B_x}$, and 
    \item for every vertex ${v \in V(G)}$, the set $\{ x \in V(T) : v \in B_x \}$ induces a non-empty (connected) subtree of $T$. 
\end{itemize}
The \defn{width} of such a $T$-decomposition is ${\max\{ |B_x| : x \in V(T) \}-1}$. A \defn{tree-decomposition} is a $T$-decomposition for any tree $T$. 
A \defn{path-decomposition} is a $P$-decomposition for any path $P$, denoted by the corresponding sequence of bags. 
The \defn{treewidth} of a graph $G$, denoted \defn{$\tw(G)$}, is the minimum width of a tree-decomposition of $G$. 
The \defn{pathwidth} of a graph $G$, denoted \defn{$\pw(G)$}, is the minimum width of a path-decomposition of $G$. By definition, $\tw(G)\leq\pw(G)$ for every graph $G$. Treewidth is the standard measure of how similar a graph is to a tree. Indeed, a connected graph has treewidth at most 1 if and only if it is a tree. It is an important parameter in structural graph theory, especially Robertson and Seymour's graph minor theory, and also in algorithmic graph theory, since many NP-complete problems are solvable in linear time on graphs with bounded treewidth. See \citep{HW17,Bodlaender98,Reed97} for surveys on treewidth. 

\section{Proof of \cref{Planar222}}
\label{Proofs}

Let \defn{$G^+$} be the graph obtained from a graph $G$ by adding one new vertex adjacent to every vertex in $G$. In any graph, a vertex $v$ is \defn{dominant} if $v$ is adjacent to every other vertex. We use the following lemma by \citet{LNW}. We include the proof for completeness.  

\begin{lem}[\citep{LNW}]
\label{JoinTreewidth}
For any graph $G$ and any vertex-partition $\{V_1,V_2\}$ of  $G$, \[G^+ \subsetsim G[V_1]^+ \StrongProd G[V_2]^+.\]
\end{lem}

\begin{proof}
Let $Q:= G[V_1]^+ \StrongProd G[V_2]^+$, where $r_i$ is the dominant vertex in $G[V_i]^+$, for $i\in\{1,2\}$. Let $r$ be the dominant vertex in $G^+$. Map $r$ to $(r_1,r_2)$ in $Q$, which is adjacent to every other vertex in $Q$. Map each vertex $v\in V_1$ to $(v,r_2)$ in $Q$. Map each vertex $w\in V_2$ to $(r_1,w)$ in $Q$. For each edge $vv'\in E(G[V_1])$, the images of $v$ and $v'$ are adjacent in $Q$. For each edge $ww'\in E(G[V_2])$, the images of $w$ and $w'$ are adjacent in $Q$. For all vertices $v\in V_1$ and $w\in V_2$, the images of $v$ and $w$ are adjacent in $Q$. Hence, the above mapping shows that $G^+$ is isomorphic to a subgraph of $Q$.   
\end{proof}

If every planar graph had a vertex-partition into two induced forests, then the Four-Colour Theorem would follow. \citet{CK69} constructed planar graphs that have no vertex-partition into two induced forests. On the other hand, \citet[Theorem~4.1]{Thomassen95} showed the following analogous result where the forest requirement is relaxed\footnote{\citet[Theorem~4.1]{Thomassen95} proved \cref{PlanarPartition} for planar triangulations, which implies the result for general planar graphs, since any subgraph of a triangle-forest is a triangle-forest. \cref{PlanarPartition} was rediscovered by 
\citet{KRU24}.}. A \defn{triangle-forest} is a graph in which every cycle is a triangle. 

\begin{lem}[\citep{Thomassen95}]
\label{PlanarPartition}
Every planar graph has a vertex-partition into two induced triangle-forests.
\end{lem}

We need the following elementary property of triangle-forests. 

\begin{lem}
\label{PSeudoForestMatching}
Every triangle-forest $G$ has a matching $M$ such that $G/M$ is a forest. 
\end{lem}

\begin{proof}
We proceed by induction on $|V(G)|$. If $|V(G)|\leq 3$ then the claim holds trivially. Now assume that $|V(G)|\geq 4$. We may assume that $G$ is connected. Suppose that $\deg_G(v)=1$ for some $v\in V(G)$. By induction, 
 $G-v$ has a matching $M$ such that $(G-v)/M$ is a forest. Since $\deg_G(v)=1$, 
 $G/M$ is a forest. Now assume that $G$ has minimum degree at least 2. Let $B$ be a leaf block of $G$. Since $G$ has minimum degree at least 2, $B$ is 2-connected. If $B$ has two non-adjacent vertices $v$ and $w$, then a cycle through $v$ and $w$ has length at least four, contradicting that $G$ is a triangle-forest. So $B$ induces a triangle. Say $V(B)=\{u,v,w\}$ where $w$ is the cut-vertex separating $\{u,v\}$ from $G-V(B)$. By induction, $G-u-v$ has a matching $M$ such that $(G-u-v)/M$ is a forest. Hence $M':=M\cup\{uv\}$ is a matching in $G$ such that $G/M'$ is a forest.  
 \end{proof}

\cref{PSeudoForestMatching,PlanarPartition} imply:

\begin{cor}
\label{PlanarMatching}
Every planar graph $G$ has a matching $M$ such that $G/M$ has a vertex-partition into two induced forests. 
\end{cor}

For an integer $k\geq 0$, a graph $G$ is \defn{$k$-apex} if $G-A$ is planar for some $A\subseteq V(G)$ with $|A|\leq k$. A graph $G$ is an \defn{apex-forest} if $G-A$ is a forest for some $A\subseteq V(G)$ with $|A|\leq 1$. Every apex-forest has treewidth at most 2. Thus, the next result implies and strengthens \cref{Planar222}. 

\begin{thm}
\label{Apex} 
Every $2$-apex graph $G$ is contained in $H_1 \StrongProd H_2 \StrongProd K_2$ for some apex-forests $H_1$ and $H_2$.
\end{thm}

\begin{proof}
We may assume that $G$ has an edge $ab$ such that $G':=G-a-b$ is planar. By \cref{PlanarMatching,JoinTreewidth}, $G'$ has a matching $M'$ such that $(G'/M')^+ \subsetsim H_1\boxtimes H_2$ for some apex-forests $H_1$ and $H_2$. Thus, for the matching $M:=M'\cup \{ab\}$ we have $G/M \subsetsim (G'/M')^+ \subsetsim H_1\boxtimes H_2$. Hence $G\subsetsim H_1\boxtimes H_2\boxtimes K_2$. Here we use the fact that if $M$ is a matching in a graph $G$, then $G\subsetsim (G/M)\boxtimes K_2$.
\end{proof}

Note that an analogous proof shows that every $k$-apex graph $G$ is contained in $H_1 \StrongProd H_2 \StrongProd K_{\max\{k,2\}}$ for some apex-forests $H_1$ and $H_2$.

\section{Proof of \cref{PlanarLowerBound}}

To prove \cref{PlanarLowerBound} it will be convenient to use the language of partitions. A \defn{partition} $\PP$ of a graph $G$ is a partition of $V(G)$, where each element of $\PP$ is called a \defn{part}. For a partition $\PP$ of a graph $G$, let \defn{$G / \PP$} be the graph obtained from $G$ by identifying the vertices in each non-empty part of $\PP$ to a single vertex; that is, $V(G / \PP)$ is the set of non-empty parts in $\PP$, where distinct parts $P_1,P_2 \in \PP$ are adjacent in $G / \PP$ if and only if  there exist $v_1 \in P_1$ and $v_2 \in P_2$ such that $v_1v_2 \in E(G)$. A partition $\PP$ of a graph $G$ is a \defn{tree-partition} if $G/\PP$ is contained in a tree, and $\PP$ is a \defn{star-partition} if $G/\PP$ is contained in a star. 

The next observation characterises when a graph is contained in $H_1 \StrongProd H_2 \StrongProd K_c$. We include the proof for completeness.

\begin{obs}[\citep{LNW}]
\label{PartitionProduct}
For any graphs $H_1,H_2$ and any $c\in\NN$, a graph $G$ is contained $H_1\StrongProd H_2 \StrongProd K_c$ if and only if $G$ has partitions $\PP_1$ and $\PP_2$ such that $G/\PP_i \subsetsim H_i$ for each $i\in\{1,2\}$, and $|A_1\cap A_2|\leq c$ for each $A_1\in\PP_1$ and $A_2\in\PP_2$. 
\end{obs}

\begin{proof}
$(\Rightarrow)$ Assume $G$ is contained $H_1\StrongProd H_2 \StrongProd K_c$. That is, there is an isomorphism $\phi$ from $G$ to a subgraph of $H_1\StrongProd H_2 \StrongProd K_c$. 
For each vertex $x\in V(H_1)$, let $A_x:=\{v\in V(G): \phi(v)\in \{x\}\times V(H_2)\times V(K_c)\}$. 
Similarly, for each vertex $y\in V(H_2)$, let $B_y:=\{v\in V(G): \phi(v)\in V(H_1)\times \{y\}\times V(K_c)\}$. 
Let $\PP_1:=\{A_x:x\in V(H_1)\}$ and $\PP_2:=\{B_y:y\in V(H_2)\}$, which are partitions of $G$. By construction, $G/\PP_i\subsetsim H_i$ for each $i\in\{1,2\}$. For $A_x\in \PP_1$ and $B_y\in \PP_2$, if $v\in A_x\cap B_y$ then $\phi(v)\in \{x\}\times \{y\}\times V(K_c)$. Thus $|A_x\cap B_y|\leq c$. 
	
$(\Leftarrow)$ Assume $G$ has partitions $\PP_1$ and $\PP_2$ such that $G/\PP_i\subsetsim H_i$ for each $i\in\{1,2\}$, and $|A_1\cap A_2|\leq c$ for each $A_1\in\PP_1$ and $A_2\in\PP_2$. Let $\phi_i$ be an isomorphism from $G/\PP_i$ to a subgraph of $H_i$. For each $A_1\in\PP_1$ and $A_2\in\PP_2$, enumerate the at most $c$ vertices in $A_1\cap A_2$, and map the $i$-th such vertex to $(\phi_1(A_1),\phi_2(A_2),i)$. This defines an isomorphism from $G$ to a subgraph of $H_1\StrongProd H_2 \StrongProd K_c$, as desired. 
\end{proof}

As illustrated in \cref{Fans}, a graph $F$ is a \defn{fan} if $F$ has a dominant vertex $v$, called the \defn{centre} of $F$, such that $F-v$ is a path. A graph $F$ is a \defn{double-fan} if $F$ has two dominant vertices $v$ and $w$, called the \defn{centres} of $F$, such that $F-v-w$ is a path. Note that every double-fan is a planar triangulation. 

\begin{figure}[!h]
\centering 
(a) \includegraphics{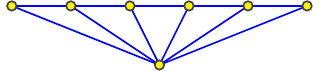} 
\qquad 
(b) \includegraphics{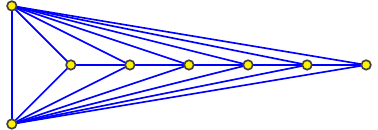} 
\caption{(a) fan, (b) double-fan}
\label{Fans}
\end{figure}

Throughout the following proofs, we use the following convention: if $\PP$ and $\QQ$ are partitions of a graph $G$, and $v_i\in V(G)$, then let \defn{$P_i$} be the part of $\PP$ with $v_i\in P_i$, and let \defn{$Q_i$} be the part of $\QQ$ with $v_i\in Q_i$. Of course, it is possible that $P_i=P_j$ or $Q_i=Q_j$ for distinct vertices $v_i,v_j$. 

\begin{lem}
\label{FanTwoParts}
For any $c\in \NN$, let $F$ be a fan on at least $c^2+c+1$ vertices with centre $v_1$. Let $\PP,\QQ$ be partitions of $F$ such that $\PP$ is a tree-partition and $|P\cap Q|\leq c$ for all $P\in \PP$ and $Q\in \QQ$. Then there exists $v_2\in V(F-v_1)$ 
such that $Q_1\neq Q_2$. 
\end{lem}

\begin{proof}
Let $B$ be the path $F-v_1$. Since $v_1$ is dominant in $F$, and $\PP$ is a tree-partition, $\PP$ is a star-partition with centre $P_1$. Since $|P_1\cap Q_1|\leq c$ and $v_1\in P_1\cap Q_1$, we have $|V(B)\cap Q_1\cap P_1|\leq c-1$. Thus $B-(Q_1\cap P_1)$ has at most $c$ components. Since $|V(B)|>(c-1)+c^2$, there is a path component  $B'$ of $B-(Q_1\cap P_1)$ on at least $c+1$ vertices. If $B'$ contains a vertex $v_2\in P_1$, then $v_2\notin Q_1$, as desired. So we may assume that $B'$ does not intersect $P_1$. Thus $B'$ is contained in a single part $P_2\in \PP$. Since $|P_2\cap Q_1|\leq c<|V(B')|$, there exists $v_2\in B'$ such that $v_2\notin Q_1$, as desired.
\end{proof}

\begin{lem}
   \label{K4inWagner}
    For any $c\in \NN$, let $F$ be a double-fan on at least $8c^2+2c+1$ vertices. Let $\PP,\QQ$ be partitions of $F$ such that $\PP$ is a tree-partition and $|P\cap Q|\leq c$ for all $P\in \PP$ and $Q\in \QQ$. Let $v_1,v_2$ be the centres of $F$. If $P_1\neq P_2$ and $Q_1\neq Q_2$, then there exist vertices $v_3,v_4$ such that $\{v_1,v_2,v_3,v_4\}$ is a $4$-clique and $Q_1,Q_2,Q_3,Q_4$ are pairwise distinct.
\end{lem}

\begin{proof}
    Let $B$ be the path $F-v_1-v_2$. 
    Since $v_1$ and $v_2$ are both adjacent to every vertex in $B$, and $v_1$ and $v_2$ are in distinct parts of $\PP$, and since $\PP$ is a tree-partition, $V(F)=P_1\cup P_2$. Since $|P\cap Q|\leq c$ for all $P\in \PP$ and $Q\in \QQ$, this means that $|Q|\leq 2c$ for each $Q\in \QQ$. In particular, since $v_1\in Q_1$ and $v_2\in Q_2$, this means that $|B\cap (Q_1\cup Q_2)|\leq 4c-2$. Thus $B-(Q_1\cup Q_2)$ has at most $4c-1$ components. Since $|V(B)|\geq 8c^2+2c-1 > (4c-2) + (4c-1)2c$, some path component  $B'$ of $B-(Q_1\cup Q_2)$ has at least $2c+1$ vertices. Since  $|Q|\leq 2c$ for each $Q\in \QQ$, $B'$ intersects at least two parts of $\QQ$. In particular, there is an edge $v_3v_4$ of $B'$ such that $Q_3\neq Q_4$. By the choice of $B'$, we have $\{Q_3,Q_4\}\cap \{Q_1,Q_2\}=\emptyset$. Thus $\{v_1,v_2,v_3,v_4\}$ is the desired $4$-clique.
\end{proof}

A \defn{distension} of a graph $G$ is any graph \defn{$\widehat{G}$} obtained from $G$ by adding, for each edge $vw$ of $G$, a path $P_{vw}$ complete to $\{v,w\}$, where $P_{vw}\cap G=\emptyset$ and $P_{vw}\cap P_{ab}=\emptyset$ for distinct $vw,ab\in E(G)$. Here `complete to' means that each vertex in $P_{vw}$ is adjacent to both $v$ and $w$. Note that $\widehat{G}[V(P_{vw})\cup\{v,w\}]$ is a double-fan, which we denote by \defn{$\widehat{G}_{vw}$}. We say $\widehat{G}$ is the \defn{$t$-distension} of $G$ if $|V(P_{vw})|=t$ for each edge $vw\in E(G)$. Observe that every distension of a planar graph is planar. 

\begin{lem}
\label{planarLBCounterExample}
For any $c\in \NN$, there exists a planar graph $G$ such that for any tree-partition $\PP$ of $G$ and for any partition $\QQ$ of $G$ with $|P\cap Q|\leq c$ for each $P\in \PP$ and $Q\in \QQ$, there is a 4-clique $\{v_1,v_2,v_3,v_4\}$ in $G$ such that $Q_1,Q_2,Q_3,Q_4$ are pairwise distinct.
\end{lem}
    
\begin{proof}
Let $t:=8c^2+2c-1$. Let $F$ be the fan on $t+1$ vertices with centre $v_1$. Let $B$ be the path $F-v_1$. Let $J$ be the $t$-distension of $F$, and let $G$ be the $t$-distension of $J$. Since $F$ is planar, $J$ is planar, and $G$ is planar. Consider any tree-partition $\PP$ of $G$ and any partition $\QQ$ of $G$ with $|P\cap Q|\leq c$ for each $P\in \PP$ and $Q\in \QQ$. 

Since $|V(F)|=t+1\geq c^2+c+1$, by \cref{FanTwoParts} applied to $F$ and the induced partitions of $F$, there exists $v_2\in B$ such that $Q_1\neq Q_2$. Let $C$ be the $t$-vertex path $J_{v_1v_2}-v_1-v_2$. 
 
Since $|V(J_{v_1v_2})|=t+2=8c^2+2c+1$, if $P_1\neq P_2$ then by \cref{K4inWagner}, there  exist vertices $v_3,v_4$ such that $\{v_1,v_2,v_3,v_4\}$ is a $4$-clique $J_{v_1v_2}$ with $Q_1,Q_2,Q_3,Q_4$ pairwise distinct, as desired. So we may assume that $P_1=P_2$.
         
Consider any vertex $v'\in V(C)$. Let $P'\in\PP$ and $Q'\in\QQ$ such that $v'\in P'\cap Q'$. Note that $v'v_1$ and $v'v_2$ are edges of $J$, so the double-fans $G_{v'v_1}$ and $G_{v'v_2}$ exist. Since $Q_1\neq Q_2$, there exists $i\in \{1,2\}$ such that $Q'\neq Q_i$. If $P'\neq P_1$ (and thus $P'\neq P_2$), since 
$|V(G_{v'v_i})|=t+2=8c^2+2c+1$, by \cref{K4inWagner} there exists a 4-clique 
$\{v_i,v',v_3,v_4\}$ in $G_{v'v_i}$ with $Q_i,Q',Q_3,Q_4$ pairwise distinct, as desired.

So we may assume that $V(C)\subseteq P_1$. Since $v_i\in P_1\cap Q_i$ for $i\in \{1,2\}$, and since $|P_1\cap Q|\leq c$ for all $Q\in \QQ$, we have $|C\cap (Q_1\cup Q_2)|\leq 2c-2$. So $C-(Q_1\cup Q_2)$ has at most $2c-1$ components. Since $|V(C)|=t=8c^2+2c-1> (2c-1)c+(2c-2)$, there exists a path component $C'$ of $C-(Q_1\cup Q_2)$ with at least $c+1$ vertices. Note that $C'\subseteq P_1$ also. Since $|P_1\cap Q|\leq c$ for all $Q\in \QQ$, there is an edge  $v_3v_4$ in $C'$ such that $Q_3\neq Q_4$. By the choice of $C'$, $\{Q_3,Q_4\}\cap \{Q_1,Q_2\}=\emptyset$. So $\{v_1,v_2,v_3,v_4\}$ is the desired $4$-clique in $G$.
\end{proof}

The next result follows from \cref{planarLBCounterExample,PartitionProduct}, which implies \cref{PlanarLowerBound}.

\begin{thm}
\label{MainLowerBound}
For any $c\in\NN$ there exists a planar graph $G$ such that for any tree $T$ and graph $H$, if $G\subsetsim H\boxtimes T \boxtimes K_c$ then $K_4\subsetsim H$ and $\tw(H)\geq 3$. 
\end{thm}

\cref{MainLowerBound} strengthens a result of 
\citet{DJMMUW20}, who proved it when $T$ is a path. 

The next lemma implies that the graph $G$ in \cref{MainLowerBound} has bounded treewidth and bounded pathwidth. In particular, since  $G$ is a distension of a distension of a fan $F$, and since $\tw(F)\leq\pw(F)\leq 2$, we have $\tw(G)\leq 3$ and $\pw(G)\leq \pw(\widehat{F})+2\leq \pw(F)+4\leq 6$. (With a more detailed analysis, one can prove \cref{MainLowerBound} with $\pw(G)\leq 4$ for a  slightly different graph $G$; we omit this result.)\ 

\begin{lem}
\label{TreewidthPathwidthBounds}
For any graph $G$ and any distension $\widehat{G}$ of $G$,
\[ \tw(\widehat{G}) \leq\max\{\tw(G),3\}
\quad\text{and}\quad
\pw(\widehat{G})\leq\pw(G)+2.\]
\end{lem}

\begin{proof}
We first prove the treewidth bound. Consider a tree-decomposition $(B_x:x\in V(T))$ of $G$ with width $\tw(G)$. Apply the following operation for each edge $vw\in E(G)$. Let $x_0\in V(T)$ such that $v,w\in B_{x_0}$. Say $P_{vw}=(u_1,\dots,u_t)$ is the path complete to $\{v,w\}$ in $\widehat{G}$. Add new vertices $x_1,\dots,x_{t-1}$ and edges $x_0x_1,x_1x_2,\dots,x_{t-2}x_{t-1}$ to $T$. Let $B_{x_i}:=\{v,w,u_i,u_{i+1}\}$ for each $i\in\{1,\dots,t-1\}$. We obtain a tree-decomposition of $\widehat{G}$, in which  every new bag has size 4. Thus $\tw(\widehat{G})\leq\max\{\tw(G),3\}$. 

We now prove the pathwidth bound. Let $(B_1,\dots,B_m)$ be a path-decomposition of $G$ with width $\pw(G)$. By duplicating bags, we may assume there is an injection $f:E(G)\to\{1,\dots,m\}$ such that $v,w\in B_{f(vw)}$ for each edge $vw\in E(G)$. For each edge $vw\in E(G)$, if $i:=f(vw)$ and $P_{vw}=(u_1,\dots,u_t)$ as above, then insert the sequence of bags $B_i\cup\{u_1,u_2\},
B_i\cup\{u_2,u_3\},\dots,B_i\cup\{u_{t-1},u_t\}$ between $B_i$ and $B_{i+1}$. 
Since $f$ is an injection, 
we obtain a path-decomposition of $\widehat{G}$, in which each bag has size at most two more than the original bag. Thus $\pw(\widehat{G})\leq\pw(G)+2$. 
\end{proof}

\subsection*{Acknowledgements} 
Distel is supported by an Australian Government Research Training Program Scholarship.
Hendrey and Wood are supported by the Australian Research Council.
Karol and Yip are supported by Monash Graduate Scholarships.

\newcommand{\papertitle}[2]{\href{#1}{\textcolor{Navy}{#2}}}

{
\fontsize{10pt}{11pt}
\selectfont
\def\soft#1{\leavevmode\setbox0=\hbox{h}\dimen7=\ht0\advance \dimen7 by-1ex\relax\if t#1\relax\rlap{\raise.6\dimen7 \hbox{\kern.3ex\char'47}}#1\relax\else\if T#1\relax \rlap{\raise.5\dimen7\hbox{\kern1.3ex\char'47}}#1\relax \else\if d#1\relax\rlap{\raise.5\dimen7\hbox{\kern.9ex \char'47}}#1\relax\else\if D#1\relax\rlap{\raise.5\dimen7 \hbox{\kern1.4ex\char'47}}#1\relax\else\if l#1\relax \rlap{\raise.5\dimen7\hbox{\kern.4ex\char'47}}#1\relax \else\if L#1\relax\rlap{\raise.5\dimen7\hbox{\kern.7ex \char'47}}#1\relax\else\message{accent \string\soft \space #1 not defined!}#1\relax\fi\fi\fi\fi\fi\fi}

}

\end{document}